\documentclass[reqno, letterpaper, oneside]{article}
\usepackage{amsmath,amsfonts,amssymb,amsthm}
\usepackage[final]{graphicx}
\usepackage[usenames,dvipsnames]{color}

\usepackage{bbm}
\usepackage{geometry}

\usepackage[colorlinks=false,,linktoc=allbookmarksopen=true,linktocpage,pdftex]{hyperref}
\usepackage[pdftex]{bookmark}

\newtheorem{theorem}{Theorem}
\newtheorem{lemma}[theorem]{Lemma}

\newtheorem{claim}[theorem]{Claim}

\theoremstyle{definition}

\newcommand{\refT}[1]{Theorem~\ref{#1}}

\newcommand{\refS}[1]{Section~\ref{#1}}

\newcommand{\refApp}[1]{Appendix~\ref{#1}}

\newcommand{\refCl}[1]{Claim~\ref{#1}}

\renewcommand\Pr{\operatorname{\mathbb P{}}}
\newcommand\E{\operatorname{\mathbb E{}}}

\newcommand\Bin{\operatorname{Bin}}

\newcommand\set[1]{\ensuremath{\{#1\}}}

\newcommand\xpar[1]{(#1)}
\newcommand\bigpar[1]{\bigl(#1\bigr)}
\newcommand\Bigpar[1]{\Bigl(#1\Bigr)}
\newcommand\biggpar[1]{\biggl(#1\biggr)}

\newcommand\Bigsqpar[1]{\Bigl[#1\Bigr]}

\newcommand\bigabs[1]{\bigl|#1\bigr|}

\def\rompar(#1){\textup(#1\textup)}    

\def\xexp(#1){e^{#1}}
\newcommand\ceil[1]{\lceil#1\rceil}
\newcommand\bigceil[1]{\bigl\lceil#1\bigr\rceil}
\newcommand\floor[1]{\lfloor#1\rfloor}

\newcommand{\cC}{\mathcal{C}}

\newcommand{\cE}{\mathcal{E}}

\newcommand{\cH}{\mathcal{H}}
\newcommand{\cI}{\mathcal{I}}

\newcommand{\cK}{\mathcal{K}}

\newcommand{\cP}{\mathcal{P}}

\newcommand{\cS}{\mathcal{S}}
\newcommand{\cT}{\mathcal{T}}

\newcommand{\cA}{\mathcal{A}}
\newcommand{\cB}{\mathcal{B}}

\newcommand{\invinv}{^{\scriptscriptstyle\text{-}1}}
\newcommand{\inv}{\invinv\subN}
\newcommand{\subN}{_{\scriptscriptstyle N}}

\newcommand{\indic}[1]{\mathbbm{1}_{\{{#1}\}}}

\let\OLDthebibliography\thebibliography
\renewcommand\thebibliography[1]{
	\OLDthebibliography{#1}
	\setlength{\parskip}{0pt}
	\setlength{\itemsep}{0pt plus 0.3ex}
}

\title{Bounds on Ramsey Games via Alterations}
\author{He Guo%
	\thanks{Faculty of Mathematics, Technion, Haifa~32000, Israel. E-mail: {\tt hguo@campus.technion.ac.il}.}
	\ and Lutz Warnke%
	\thanks{Department of Mathematics, University of California, San Diego, La Jolla CA~92093, USA. 
		E-mail: {\tt lwarnke@ucsd.edu}. 
		Supported by NSF~grant DMS-1703516, NSF~CAREER grant~DMS-1945481, and a Sloan Research Fellowship.}
}
\date{January 13, 2022; revised January 24, 2023}
\begin{document}
	\maketitle
	
\thispagestyle{empty}

\begin{abstract}
	We present a refinement of the classical alteration method for constructing $H$-free graphs: 
	for suitable edge-probabilities~$p$, we show that removing all edges in $H$-copies of the binomial random graph~$G_{n,p}$  
	does not significantly change the independence number. 
	This differs from earlier alteration approaches of Erd{\H{o}}s and Krivelevich, 
	who obtained similar guarantees by removing one edge from each $H$-copy (instead of all of~them). 
	We demonstrate the usefulness of our refined alternation method via two applications to online graph Ramsey games, where it enables easier analysis.   
\end{abstract}


\section{Introduction}
The probabilistic method is a widely-used tool in discrete mathematics. 
Many of its powerful approaches have been developed in the pursuit of understanding the graph Ramsey number~$R(H,k)$, 
which is defined as the the minimum number~$n$ so that any $n$-vertex graph contains either a copy of~$H$ or an independent set of size~$k$. 
For example, in~1947 Erd{\H{o}}s pioneered the \emph{random coloring approach} 
to obtain the lower bound~$R(K_k,k) = \Omega(k 2^{k/2})$ as~$k \to \infty$, where~$K_k$ denotes a~$k$-vertex clique, 
and in~1961 he developed the \emph{alteration method} in order 
to obtain~$R(K_3,k)=\Omega(k^2/(\log k)^2)$ as~$k \to \infty$, see~\cite{erdos1961graph}. 
In~1975 and~1977 Spencer~\cite{Spencer1975,Spencer1977} reproved these 
results via the \emph{Lov\'asz Local Lemma}, 
and also extended them to lower bounds on~$R(H,k)$ for~$H \in \{K_s,C_{\ell}\}$, 
where~$C_{\ell}$ denotes a cycle of length~$\ell$. 
In~1994 Krivelevich~\cite{krivelevich1995bounding} further extended this to general graphs~$H$ via a new \emph{(large-deviation based) alteration approach}, 
obtaining the lower~bound 
\begin{equation}\label{eq:RamseyHk}
	R(H,k)=\Omega\Bigpar{(k/\log k)^{m_2(H)}}
	\quad \text{ with } \quad m_2(H):=\max_{F\subseteq H}\Bigpar{\indic{v_F \ge 3}\tfrac{e_F-1}{v_F-2} + \indic{F=K_2}\tfrac{1}{2}} 
\end{equation}
as~$k \to \infty$, 
where the implicit constants may depend on~$H$ 
(writing~$v_F:=|V(F)|$ and~$e_F := |E(F)|$ for the number of vertices and edges of a graph~$F$, 
and denoting by~$\indic{\cE}$ the indicator variable for the event~$\cE$, as usual). 
By analyzing \emph{(semi-random) $H$-free processes}, 
in~1995 Kim~\cite{kim1995ramsey}  and in~2010 Bohman--Keevash~\cite{bohman2010early} have further 
improved the logarithmic factors in~\eqref{eq:RamseyHk} for some forbidden graphs~$H$ such as 
triangles~$K_3$, cliques~$K_s$, and cycles~$C_\ell$. 
In~2019 Mubayi and Verstraete~\cite{MV2019} used \emph{pseudorandom graphs} to obtain further 
logarithmic improvements of~\eqref{eq:RamseyHk} for cycles~$C_\ell$ of odd length~$\ell \ge 5$.  
However, despite considerable effort, for~$H \neq K_3$ the best known lower and upper bounds 
on the graph Ramsey number~$R(H,k)$ are still polynomial factors apart, see~\cite{bohman2010early,bohman2013dynamic,pontiveros2013triangle}. 
%
Unsurprisingly, in quest for progress research has thus stretched in several directions, 
including (i)~variants of Ramsey numbers and (ii)~further advancements of existing probabilistic proof methods.

In this paper we present a \emph{refinement of the alteration method} for constructing $H$-free graphs: 
for suitable edge-probabilities~$p=p(n)$, we show that removing all edges in \mbox{$H$-copies} of the binomial random graph~$G_{n,p}$  
does not significantly change the independence number, i.e., the size of the largest independent set (see \refS{sec:intro:tool}). 
The main innovation is that we can allow for deleting all edges in \mbox{$H$-copies}, 
and not just one edge from each \mbox{$H$-copy} as in earlier alteration approaches of Erd{\H{o}}s~\cite{erdos1961graph} and Krivelevich~\cite{krivelevich1995bounding}.
Our refinement is a natural random graph statement in its own right, 
and we demonstrate its usefulness via two applications to online graph Ramsey games, where it enables easier analysis. 
In particular, for Ramsey,~Paper,~Scissors~games and online~Ramsey~numbers this 
allows us to extend bounds of
Conlon, Fox, Grinshpun and~He~\cite{conlon2018online} and Fox, He and Wigderson~\cite{fox2019ramsey} 
to a large class of forbidden graphs~$H$ (see~\refS{sec:intro:application}).

\subsection{Refined alteration method}\label{sec:intro:tool}
To motivate our refined alteration approach, we first review the classical alteration arguments for the Ramsey bound~\eqref{eq:RamseyHk}
due to Erd{\H{o}}s~\cite{erdos1961graph} and Krivelevich~\cite{krivelevich1995bounding}. 
They both use a 
binomial random graph~$G_{n,p}$ with ${n=\Theta((k/\log k)^{m_2(H)})}$ vertices and edge-probability ${p=\Theta((\log k)/k)}$ 
to construct an $n$-vertex subgraph ${G \subseteq G_{n,p}}$ with two properties: 
(i)~$G$ is $H$-free, and 
(ii)~$G$ contains no independent set~$K$ of size~$k$; 
both properties together imply~${R(H,k)> n}$. 
Chernoff bounds suggest that the number~$X_K:=|E(G_{n,p}[K])|$ of edges of $G_{n,p}$ inside~$K$ is around~$\tbinom{k}{2}p$, 
so for property~(ii) it suffices to show that the alteration from~$G_{n,p}$ to~$G$ does not remove `too many' edges from each $k$-vertex subset~$K$. 
To illustrate that this is a non-trivial task, 
let us consider the natural upper bound~$e_H \cdot |\cH_K|$ on the number of removed edges from~$K$, 
where~$\cH_K$ denotes the collection of all $H$-copies that have at least one edge inside~$K$. 
For any~$\delta>0$ it turns out that~$\Pr(|\cH_K| \ge \delta \tbinom{k}{2}p) \ge e^{-o(k)}$ due to `infamous' upper tail~\cite{JR2002,SW2018} behavior (see~\refApp{apx:UT}), 
which 
suggests that one has to rather carefully handle edges that are contained in multiple~$H$-copies.

For triangles~$H=K_3$,  Erd{\H{o}}s~\cite{erdos1961graph} overcame these difficulties in~1961
by a clever ad-hoc greedy alteration argument, 
showing that whp\footnote{In this paper \emph{whp} (with high probability) always means with probability tending to one as $k \to \infty$.} (with high probability) the following works: 
If one sequentially traverses the edges of~$G_{n,p}$ in any order, only accepting edges that do not create a triangle together with previously accepted edges, then the resulting `accepted' subgraph~$G \subseteq G_{n,p}$ satisfies the independent set property~(ii), and trivially the \mbox{$H$-free} property~(i). 
The fact that any edge-order works was exploited by 
Conlon, Fox, Grinshpun and~He~\cite{conlon2018online} and 
Fox, He and Wigderson~\cite{fox2019ramsey} 
in their analysis of triangle-based online Ramsey~games. 

To handle general graphs~$H$, Krivelevich~\cite{krivelevich1995bounding} developed in~1994 an elegant alteration~argument, 
showing that whp the following works: 
If one constructs~${G \subseteq G_{n,p}}$ by deleting all edges that are in some maximal (under inclusion) collection~$\cC$ of edge-disjoint~$H$-copies in~$G_{n,p}$, 
then this (a)~removes less than~$X_K \approx \tbinom{k}{2}p$ edges from each $k$-vertex subset~$K$, and (b)~yields an $H$-free graph by maximality of~$\cC$, 
establishing both properties~(ii) and~(i).  
Unfortunately, this slick 
argument is hard to adapt to online Ramsey games, 
where one cannot foresee whether in future turns a given edge will be contained in an~$H$-copy or not (cf.~\refS{sec:intro:application}).

Our refined alteration approach overcomes the discussed difficulties,  
by showing that whp the desired \mbox{$H$-free} and independent set properties~(i) and~(ii) remain valid even 
if one deletes \emph{all} edges from~$G_{n,p}$ that are in some \mbox{$H$-copy}. 
\refT{thm:Hremoval} is stated for strictly \mbox{$2$-balanced} graphs~$H$, 
i.e., which satisfy ${m_2(H)>m_2(F)}$ for all subgraphs~${F\subsetneq H}$ 
(this class includes many graphs of interest,
including cliques~$K_s$, cycles~$C_\ell$, 
complete multipartite graphs~$K_{t_1, \ldots, t_r}$, and hypercubes~$Q_d$). 
For many applications such as~\eqref{eq:RamseyHk} the restriction to strictly \mbox{$2$-balanced} graphs is immaterial, 
since one can often obtain the general case by forbidding a strictly \mbox{$2$-balanced} subgraph~$H_0 \subseteq H$ with $m_2(H_0)=m_2(H)$; cf.~\refS{sec:intro:OnlineRamseyGame}.
Below~$X_K=|E(G_{n,p}[K])|$ denotes the number of edges of $G_{n,p}$ inside~$K$, 
and~$Y_K$ denotes the number of edges in~$E(G_{n,p}[K])$ that are in some~$H$-copy of~$G_{n,p}$. 
%
\begin{theorem}[Main alteration tool]\label{thm:Hremoval}%
Let~$H$ be a strictly~$2$-balanced graph. 
Then, for any fixed~${\delta>0}$, the following holds for all sufficiently large~${C \ge C_0=C_0(\delta,H)}$ and sufficiently small~${0 < c \le c_0=c_0(C,\delta,H)}$. 
The random graph~$G_{n,p}$ with ${n := \floor{c (k/\log k)^{m_2(H)}}}$ vertices and edge-probability ${p := C(\log k)/k}$ 
whp 
satisfies~${Y_K \le \delta \tbinom{k}{2}p}$ and~${X_K \ge (1-\delta) \tbinom{k}{2}p}$ for all~$k$-vertex subsets~$K$ of~$G_{n,p}$. 
\end{theorem}
\noindent
As discussed, our refined alteration method constructs ${G \subseteq G_{n,p}}$ 
by deleting all edges that are in some \mbox{$H$-copy} of~$G_{n,p}$, 
so that~$G$ trivially satisfies the \mbox{$H$-free} property~(i). 
For fixed~$ \delta \in (0,1/2)$ and suitable~$n,p$, our main 
tool \refT{thm:Hremoval} then ensures 
that the following holds whp: for any $k$-vertex subset~$K$ of~$G$ we~have 
\[
|E\bigpar{G[K]}| = {X_K-Y_K} \ge {(1-2\delta) \tbinom{k}{2}p} > 0 ,
\]
which implies that~$G$ also satisfies the independent set property~(ii).  
To put this into context, we remark that the size of the largest independent set of~$G_{n,p}$ whp satisfies $\alpha(G_{n,p}) = \Theta(\log(np)/p)= \Theta(k)$ when~$m_2(H)>1$, see~\cite[Section~7.1]{JLR}, 
which implies that whp~$\Theta(k) = \alpha(G_{n,p}) \le \alpha(G) < k$. 
The conceptual crux is thus that deleting all edges in~$H$-copies does not significantly change the independence number of~$G_{n,p}$ 
(earlier alteration approaches~\cite{erdos1961graph,krivelevich1995bounding} only gave such guarantees 
after deleting a carefully chosen subset of these edges). 

As we shall see in~\refS{sec:intro:application}, variants of the above-discussed alteration argument carry over to certain online Ramsey games 
(where arbitrary deletion of edges in \mbox{$H$-copies} will be key). 
We believe that the upper bound on~$Y_K$ from \refT{thm:Hremoval} will also be useful in other contexts, 
since the the classical alteration method has become a standard tool in many applications, 
including coloring problems for graphs with forbidden substructures~\cite{krivelevich1997minimal,AF2008,bohman2010coloring},
Ramsey number variants~\cite{krivelevich1995bounding,Sudakov2007, conlon2018online,fox2019ramsey}, 
induced bipartite triangle-free graphs~\cite{erdos1988,KLST2018}, 
algorithmic approximation ratios~\cite{krivelevich1997approximate}, 
and semi-random constructions~\cite{kim1995ramsey,GW2017,GW2021}.
Further extensions of our refined alteration method are discussed in \refS{sec:extensions}.

\section{Applications to online Ramsey games}\label{sec:intro:application}
In this section we demonstrate the usefulness of our refined alteration method via two applications to online graph Ramsey games. 
Here it will be crucial that we can allow for arbitrary deletion of edges in $H$-copies, 
which enables easier analysis in the online setting (by similar reasoning as in the classical `offline'~setting). 

\subsection{Ramsey, Paper, Scissors Game}\label{sec:intro:RamseyPaperScissors}
Our first application concerns the \emph{Ramsey, Paper, Scissors game} 
that was introduced by Fox, He and~Wigderson~\cite{fox2019ramsey}. 
For a graph~$H$, this is a game between two players, Proposer and Decider, 
that starts with a finite set~$V={\{1,2, \ldots, n\}}$ of~$n$ isolated vertices. 
In each turn, Proposer proposes a pair of non-adjacent vertices from~$V$, 
and Decider simultaneously decides whether or not to add it as an edge to the current graph (without knowing which pair is proposed). 
Proposer cannot propose vertex-pairs that would form a copy of~$H$ together the current graph, nor vertex-pairs that have been proposed before. 
The {\emph{RPS~number}~$\mathrm{RPS}(H,n)$} is defined\footnote{For imperfect-information games such as Ramsey, Paper, Scissors (both players make simultaneous moves) one usually considers randomized strategies, see~{\cite[pp.~14,~169]{GameTheory}}, 
	motivating why the definition of~$\mathrm{RPS}(H,n)$ includes probability~of~winning.}  
as the largest number~$k$ for which Proposer can guarantee that, 
with probability at least~$1/2$ (regardless of Decider's strategy), 
the final graph has an independent set of size~$k$. 
Our refined alteration method enables us to prove the following upper bound on~$\mathrm{RPS}(H,n)$ 
for all strictly $2$-balanced graphs~$H$.
%
\begin{theorem}[Ramsey, Paper, Scissors Game]\label{thm:main2}
If~$H$ is a strictly $2$-balanced graph, 
then the RPS~number satisfies ${\mathrm{RPS}(H,n) = O(n^{1/m_2(H)}\log n)}$ 
as~$n \to \infty$, where the implicit constant may depend on~$H$. 
\end{theorem}
\noindent
For all strictly $2$-balanced graphs~$H$, \refT{thm:main2} gives the best known upper bounds for RPS numbers.
For $s$-vertex cliques we obtain~$\mathrm{RPS}(K_s,n)=O\bigpar{n^{2/(s+1)} \log n}$, 
which generalizes the upper bound part of the~$\mathrm{RPS}(K_3,n)=\Theta(\sqrt{n} \log n)$ result of Fox, He and~Wigderson~\cite{fox2019ramsey} for triangles. 
%

The following proof of \refT{thm:main2} demonstrates that for Ramsey, Paper, Scissors games, 
our refined alteration approach leads to conceptually simple upper bound proofs.  
%
\begin{proof}[Proof of \refT{thm:main2}]
For~$\delta := 1/4$ we choose~$C>0$ large enough and then~$c>0$ small enough so that~\refT{thm:Hremoval} 
applies to~$G_{n,p}$ with ${n := \floor{c (k/\log k)^{m_2(H)}}}$ and ${p := C (\log k)/k}$. 
We shall analyze the following Decider strategy: 
in each turn Decider accepts the (unknown) proposed vertex-pair as an edge independently with probability~$p$. 
To prove~$\mathrm{RPS}(H,n) < k = O(n^{1/m_2(H)}\log n)$ as~$n \to \infty$ (where the implicit constant depends on~$H$), 
it suffices to show that the resulting final graph has whp no independent set of size~$k$. 

Turning to details, let~$G$ denote the resulting final graph at the end of the game, i.e., which contains all accepted edges.  
Since all edges that do not create~$H$-copies are eventually proposed, 
there is a natural coupling between~$G_{n,p}$ and~$G$ which satisfies the following two properties: 
(a)~that~${E(G) \subseteq E(G_{n,p})}$, and
(b)~that every edge in~${E(G_{n,p}) \setminus E(G)}$ is contained in an~$H$-copy of~$G_{n,p}$. 
Invoking \refT{thm:Hremoval}, it follows 
that this coupling satisfies the following whp: 
for any $k$-vertex subset~$K$ of~$G$ we~have 
\[
\bigabs{E\bigpar{G[K]}} \: \ge \: X_K-Y_K \: \ge \: (1-2\delta) \tbinom{k}{2}p = \tfrac{1}{2}\tbinom{k}{2}p > 0 , 
\]
which implies that the final graph~$G$ has whp no independent set of size~$k$, 
so that the desired upper bound~$\mathrm{RPS}(H,n) < k= O(n^{1/m_2(H)}\log n)$ follows (as discussed~above).  
\end{proof}
%

\subsection{Online Ramsey Game}\label{sec:intro:OnlineRamseyGame}
Our second application concerns the widely-studied \emph{online Ramsey game} 
(see, e.g.,~\cite{Beck,KR2005,KK2009,Conlon2009,conlon2018online}) that was 
introduced independently by Beck~\cite{Beck} and Kurek--Ruci{\'n}ski~\cite{KR2005}. 
This is a game between two players, Builder and Painter, 
that starts with an infinite set~${V=\{1,2, \ldots\}}$ of isolated vertices. 
In each turn, Builder places an edge between two non-adjacent vertices from~$V$, 
and Painter immediately colors it either red or blue. 
The {\emph{online Ramsey number}~$\tilde{r}(H,k)$} is defined 
as the smallest number of turns~$N$ that Builder needs to guarantee the existence of 
either a red copy of~$H$ or a blue copy of~$K_k$ (regardless of Painter's strategy). 
Our refined alteration method enables us to prove the following lower bound on~$\tilde{r}(H,k)$, 
which, up to logarithmic factors, is about~$k$ times the best-known general 
lower bound for the usual Ramsey number~$R(H,k)$; cf.~\eqref{eq:RamseyHk}.
\begin{theorem}[Online Ramsey Game]\label{thm:maincor}
If~$H$ is a graph with~$e_H \ge 1$ edges, then the online Ramsey number satisfies  
${\tilde{r}(H,k)=\Omega\bigpar{k\cdot (k/\log k)^{m_2(H)}}}$ as~$k \to \infty$, 
where the implicit constant may depend on~$H$. 
\end{theorem}
\noindent
For general graphs~$H$, \refT{thm:maincor} gives the best known lower bounds for online Ramsey numbers. 
For \mbox{$s$-vertex} cliques we obtain~$\tilde{r}(K_s,k)=\Omega\bigpar{k^{(s+3)/2}/(\log k)^{(s+1)/2}}$, 
which generalizes the lower bound of Conlon, Fox, Grinshpun and He~{\cite[Theorem~4]{conlon2018online}} for triangles, 
and improves~{\cite[Corollary~3]{conlon2018online}} for small clique sizes~$s$. 
%

The following proof of \refT{thm:maincor} demonstrates that for Online Ramsey games, 
our refined alteration approach again leads to proofs that conceptually mimic the usual reasoning from the classical offline setting. 
To this end we shall analyze a semi-random Painter strategy, whose main ideas we now outline. 
The default color of an edge is blue. 
But if an edge is placed between vertices of sufficiently `high' degree, it does the following independently with probability~$p$: 
it colors the edge red, unless this would create a red $H$-copy, in which case the edge is still colored blue (i.e., the red-coloring attempt is rejected). 
By construction there are no red $H$-copies, and 
it turns out that blue cliques~$K_k$ can only appear inside the growing set~$U$ of high-degree vertices 
(since vertices in~$K_k$ must have `high'~degree). 
Ignoring a number of technicalities, 
after the first~$N=\Theta\bigpar{k\cdot (k/\log k)^{m_2(H)}}$ turns, 
we are able to establish that, inside each~$k$-vertex subset~$K \subseteq U$, 
the following holds: 
(i)~the number of `rejected' red-coloring attempts is at most~$Y_K$, 
and (ii)~the total number of red-coloring attempts is at least~$X^{\star}_{K} \sim {\Bin\bigpar{\tfrac{1}{2}\tbinom{k}{2},p}}$. 
Using our refined alteration approach and Chernoff Bounds we can then  show that, whp, 
every~$k$-vertex subset~$K \subseteq U$ contains at least~$X^{\star}_{K}-Y_K > 0$ red~edges, 
which prevents blue cliques~$K_k$ and thus gives~$\tilde{r}(H,k) > N=\Theta\bigpar{k\cdot (k/\log k)^{m_2(H)}}$. 
The technical details of the following proof are complicated by the fact that the online Ramsey game is played on an infinite set ${V=\{1,2, \ldots\}}$ of~vertices, 
which requires special care in the coupling and union bound~arguments. 
\begin{proof}[Proof of \refT{thm:maincor}]
For convenience we first suppose that~$H$ is strictly~$2$-balanced. 
For~$\delta := 1/8$ we choose~${C \ge 64 e_H}$ large enough and then~$c>0$ small enough 
so that~\refT{thm:Hremoval} applies to~$G_{n,p}$ with ${n := \floor{c (k/\log k)^{m_2(H)}}}$ and ${p := C (\log k)/k}$.  
Set~${L :=\floor{(k-1)/4}}$. 
At any moment of the game, we define~$U \subseteq V$ as the set of all vertices that, 
in the current graph, are adjacent to at least~$L$ edges placed by builder 
(to clarify: the growing vertex set~$U$ is updated at the end of each~turn).

We shall analyze the following Painter strategy: Painter's default color is blue, 
but if an edge~$e=\{x,y\}$ is placed inside~$U$, then Painter does the following independently with probability~$p$~($\star$): 
it colors the edge~$e$ red, unless this would create a red $H$-copy~($\dagger$), 
in which case the edge~$e$ is still colored blue. 
By construction there are no red $H$-copies, and blue cliques~$K_k$ can only appear inside~$U$ 
(since all vertices in copy of~$K_k$ must be adjacent to at least~$k-1 > L$~vertices).
To prove~$\tilde{r}(H,k) > N:= \floor{L \cdot n /2} = \Omega(k \cdot (k/\log k)^{m_2(H)})$ as~$k \to \infty$ (with implicit constants depending on~$H$), 
by the usual reasoning it remains to show that after~$N$~turns there are whp no blue cliques~$K_k$ inside~$U$. 
Let~$\cK$ denote the collection of all $k$-vertex subsets~$K \subseteq U$ after~$N$~turns. 
The plan is to show that, inside each vertex set~$K \in \cK$ that can become a blue clique~$K_k$, 
there are more red-coloring attempts~($\star$) than `rejected' red-coloring attempts~($\dagger$), 
which enforces a red edge~inside~$K$.

Turning to details, note that~$|U| \le 2N/L \le n$ during the first~$N$ turns. 
Using the order in which vertices enter~$U$ (breaking ties using lexicographic order), 
at any moment during the first~$N$ turns we thus obtain an injection~$\Phi:U \mapsto \{1, \ldots, n\} = V(G_{n,p})$. 
After~$N$ turns, we abbreviate this injection by~$\Phi\subN$, and write~$\Phi\subN(K) := \{ \Phi\subN(v): v \in K\}$. 
Define~$\cB_K$ as the event that, during the first~$N$ turns, the number of `rejected' red-coloring attempts~($\dagger$) inside~$K$ is at most~$\tfrac{1}{8} \tbinom{k}{2}p$. 
There is a natural turn-by-turn inductive coupling between~$G_{n,p}$ and Painter's strategy, where the red-coloring attempt~($\star$) occurs if~$\Phi(e) :=\{\Phi(x),\Phi(y)\}$ is an edge of~$G_{n,p}$. 
A moments thought reveals that, during the first~$N$ turns, under this coupling the total number of `rejected' red-colorings~($\dagger$) inside~$K \in \cK$ is at most~$Y_{\Phi\subN(K)}$ defined with respect to~$G_{n,p}$ 
(since~($\dagger$) can only happen when a red-coloring of~$e \subseteq K$ creates a red~$H$-copy, 
which under the coupling implies that~$\Phi(e) \subseteq \Phi(K)$ is contained in an~$H$-copy of~$G_{n,p}$). 
Applying \refT{thm:Hremoval} with~$\delta =1/8$ to~$G_{n,p}$,
using the described coupling and~$|\Phi\subN(K)|=|K|=k$ it then follows that, whp, the event~$\cB_K$ occurs for all~$K \in \cK$.

Intuitively, we shall next show that, for all $k$-vertex sets~$K \in \cK$ that contain~$\binom{k}{2}$ edges (a prerequisite for having a blue clique~$K_k$ inside~$K$), 
the total number of red-coloring attempts~($\star$) inside~$K$ is at least~$\tfrac{1}{4} \tbinom{k}{2}p$. 
To make this precise, define~$\cT_K$ as the event that builder places less than~$\tbinom{k}{2}$ edges inside~$K$ during the first~$N$ turns.
Let~$X^{\star}_{K}$ denote the number of red-coloring attempts~($\star$) inside~$K$ during the first~$N$ turns, 
and define~$\cA_K$~as the event that~$X^{\star}_{K} \ge \tfrac{1}{4} \tbinom{k}{2}p$. 
Let~$\cK'$~denote the collection of all $k$-vertex subsets~$K' \subseteq V(G_{n,p})$. 
Since~$\Phi\subN$ defines an injection from~$\cK$~to~$\cK'$, writing~$\Phi\inv(K') := \{v \in V: \Phi\subN(v) \in K'\}$ 
it follows that 
\begin{equation}\label{eq:TkAK}
	\Pr(\text{$\neg\cA_K \cap \neg\cT_K$ for some $K \in \cK$}) 
	\;\le\; \sum\nolimits_{K' \in \cK'} \Pr\Bigpar{\text{$X^{\star}_{\Phi\inv(K')} \le \tfrac{1}{4}\tbinom{k}{2}p\:$~and~$\;\neg\cT_{\Phi\inv(K')}$}}. 
\end{equation}
Fix~$K' \in \cK'$, and set~$K := \Phi\inv(K')$. 
Note that, by checking in each turn for red-coloring attempts~($\star$) inside $\Phi\invinv(K') := \{v \in V: \Phi(v) \in K'\}$, 
we can determine~$X^{\star}_{K}$ without knowing~$\Phi\inv$ in advance. 
Furthermore, since every vertex is adjacent to at most~$L$ vertices before entering~$U$, 
the event~$\neg\cT_{K}$ implies that during the first~$N$ turns at 
least~${\tbinom{k}{2}-|K| \cdot L \ge \tfrac{1}{2}\tbinom{k}{2}}$ red-coloring attempts~($\star$)
happen inside~$K$, each of which is (conditional on the history) successful with probability~$p$. 
It follows that~$X^{\star}_{K}$ stochastically dominates 
a binomial random variable~$Z \sim\mathrm{Bin}\bigpar{\bigceil{\tfrac{1}{2}\binom{k}{2}},p}$, 
unless the event~$\cT_{K}$ occurs. 
Noting~$kp = C\log k \ge 64 e_H \log k$ and~$n \ll k^{e_H}$, 
by invoking standard Chernoff bounds (such as~\cite[Theorem~2.1]{JLR}) 
it then follows that 
\begin{equation}\label{eq:XK:Chernoff:dom}
	\Pr\Bigpar{\text{$X^{\star}_{\Phi\inv(K')} \le \tfrac{1}{4}\tbinom{k}{2}p\:$~and~$\;\neg\cT_{\Phi\inv(K')}$}}
	\;\le\;
	\Pr\Bigpar{Z \le \tfrac{1}{4}\tbinom{k}{2}p} \;\le\;
	\exp\Bigpar{-\tbinom{k}{2}p/16} \ll k^{-e_H k} \ll n^{-k}. 
\end{equation}
Combining~\eqref{eq:TkAK}--\eqref{eq:XK:Chernoff:dom} with~$|\cK'| \le n^k$, we readily infer that, whp, the event~$\cA_K \cup \cT_K$ occurs for all~$K \in \cK$.

To sum up, the following holds whp after~$N$ turns: every~$k$-vertex subset~$K \subseteq U$ contains 
either (a)~at least~${\tfrac{1}{4}\tbinom{k}{2}p-\tfrac{1}{8}\tbinom{k}{2}p=\tfrac{1}{4}\tbinom{k}{2}p > 0}$ red edges, 
or (b)~less than~$\tbinom{k}{2}$ edges in total. 
Both possibilities prevent a blue clique~$K_k$ inside~$K$, 
and so the desired lower bound~$\tilde{r}(H,k)>N$ follows 
(as discussed~above).

Finally, in the remaining case where~$H$ is not strictly~$2$-balanced, 
we pick a minimal subgraph~${H_0 \subsetneq H}$ with~${m_2(H_0)=m_2(H)}$. 
It is straightforward to check that, by construction, $H_0$~is strictly $2$-balanced. 
Furthermore, since any $H_0$-free graph is also~$H$-free, we also have~${\tilde{r}(H,k) \ge \tilde{r}(H_0,k)}$. 
Repeating the above proof with~$H$ replaced by~$H_0$ then gives the claimed lower bound on~$\tilde{r}(H,k)$. 
\end{proof}

It would be interesting to investigate whether \refT{thm:maincor} can be improved if one replaces our $G_{n,p}$ based  
alteration~approach by an {$H$-free process~\cite{bohman2010early}} based approach or semi-random~variants thereof~\cite{kim1995ramsey,GW2017}.


%
\section{Proof of main alteration tool \refT{thm:Hremoval}}\label{sec:Alteration}
In this section we prove the main tool \refT{thm:Hremoval} of our refined alteration method. 
The main difficulty is the desired upper bound on~$Y_K$, which denotes the number of edges in~$E(G_{n,p}[K])$ that are in some~$H$-copy of~$G_{n,p}$.  
Here our core proof strategy is to approximate~$Y_K$ by more tractable auxiliary random variables,  
inspired by ideas from~\mbox{\cite{JR2002,Warnke2017, warnke2016missing, vsileikis2019upper}}. 
In~particular, we expect that the main contribution to~$Y_K$ should come from $H$-copies that share exactly two vertices and one edge with~$K$; 
in the below proof we will denote the collection of such `good' $H$-copies by~$\cH_K^*$. 
Note that when multiple good~$H$-copies from~$\cH_K^*$ contain some common edge~$f$ inside~$K$, 
they together only contribute one edge to~$Y_K$. 
It follows that, by arbitrarily selecting one `representative' copy~$H_f \in \cH_K^*$ for each relevant edge~$f$, 
we should obtain a sub-collection~$\cH \subseteq \cH_K^*$ of good $H$-copies with~$|\cH| \approx Y_K$. 
The~$H$-copies in~$\cH$ share no edges inside~$K$ by construction, 
and it turns out that all other types of edge-overlaps are `rare', 
i.e., make a negligible contribution to~$Y_K$.  
We thus expect that there is an edge-disjoint sub-collection~$\cH' \subseteq \cH \subseteq \cH_K^*$ of good $H$-copies with~$|\cH'| \approx |\cH| \approx Y_K$, 
and here the crux is that the upper tail of~$|\cH'|$ is much easier to estimate than the upper tail of~$Y_K$ (see~\refCl{cl:UT}~below).
The following proof implements a rigorous variant of the above-discussed heuristic~ideas for bounding~$Y_K$. 
\begin{proof}[Proof of the~\mbox{$Y_K$ bound} of \refT{thm:Hremoval}]
Noting that the claimed bounds are trivial when~$m_2(H) \le 1$
(since then there are~no $k$-vertex subsets~$K$ in~$G_{n,p}$ due to~$n \ll k$), 
we may henceforth assume~$m_2(H)>1$. 

Fix a $k$-vertex set~$K$.
Let~$\cH_K$ denote the collection of all $H$-copies in~$G_{n,p}$ that have at least one edge inside~$K$, 
and let~$\cH^{*}_K \subseteq \cH_K$ denote the sub-collection of $H$-copies that moreover share exactly two vertices with~$K$.   
Let~$\cI_K$ denote a size-maximal collection of edge-disjoint~${H \in \cH^{*}_K}$.
Clearly~${|\cI_K| \le Y_K}$, and \refCl{claim:2} below establishes a related upper bound. 
Let~$\cT_K$ denote a size-maximal collection of edge-disjoint~${H \in \cH_K \setminus \cH^{*}_K}$. 
Let~$\cP_K$ denote a size-maximal collection of edge-disjoint~$H_1 \cup H_2$ with distinct ${H_1,H_2 \in \cH^{*}_K}$ that satisfy~$|E(H_1) \cap E(H_2)| \ge 1$ and~$V(H_1) \cap K \neq V(H_2) \cap K$. 
Let~$\Delta_{H,f}$~denote the number of $H$-copies in~$G_{n,p}$ that contain the edge~$f$, and 
define~$\Delta_H$~as the maximum of~$\Delta_{H,f}$ over all~$f \in E(K_n)$. 
\begin{claim}\label{claim:2}
We have~$Y_K \le |\cI_K| + 2 e_H^2 (|\cT_K|+|\cP_K|) \Delta_H$. 
\end{claim}
\begin{proof}[Proof of \refCl{claim:2}]
We divide the $H$-copies in~$\cH_K$ into two disjoint groups: those which share at least one edge with some~$H \in \cT_K$ or~$H_1 \cup H_2 \in \cP_K$, and those which do not; we denote these two groups by~$\cH_1$ and~$\cH_2$, respectively. 
For~$j \in \{1,2\}$, let~$\cE_j$ denote the collection of edges from~$K$ that are contained in at least one~$H$-copy from~$\cH_j$. 
Note that~$Y_K \le |\cE_1|+|\cE_2|$ and~$|\cE_1| \le e_H|\cH_1| \le e_H \cdot (e_H|\cT_K|+2e_H|\cP_K|) \Delta_H$. 

Gearing up towards bounding~$|\cE_2|$, recall that~$\cH_2$ contains all $H$-copies in~$\cH_K$ that are edge-disjoint from all~$H \in \cT_K$ and~$H_1 \cup H_2 \in \cP_K$. 
By size-maximality of~$\cT_K$, it follows that~$\cH_2 \subseteq \cH^{*}_K$ (since any $H$-copy in~$\cH_2$ that shares more than two vertices with~$K$ could be added to~$\cT_K$, contradicting maximality). 
By size-maximality of~$\cP_K$, we then infer the following property~($\bullet$) of~$\cH_2$: any two distinct $H$-copies in~$\cH_2$ are edge-disjoint, unless they both intersect~$K$ in the same two vertices (since otherwise we could add them to~$\cP_K$, contradicting maximality). 
For each $f \in \cE_2 \subseteq \tbinom{K}{2}$ we now arbitrarily select one~$H$-copy from~$\cH_2$ that contains~$f$. 
By property~($\bullet$) and the size-maximality of~$\cI_K$, 
this yields a sub-collection~$\cH'_2 \subseteq \cH_2  \subseteq \cH^{*}_K$ of edge-disjoint $H$-copies satisfying~$|\cE_2|=|\cH'_2| \le |\cI_K|$, and the claim~follows. 
\end{proof}

The remaining upper tail bounds for~$|\cI_K|$, $|\cT_K|$, $|\cP_K|$ and~$\Delta_H$ hinge on the following four key estimates.
First, $m_2(H)>1$ and strictly $2$-balancedness of~$H$ imply~${m_2(H)=(e_H-1)/(v_H-2)}$, so that 
\begin{equation}\label{eq:aux1}
	n^{v_H-2} p^{e_H-1} = \bigpar{np^{m_2(H)}}^{v_H-2} \le \bigpar{c C^{m_2(H)}}^{v_H-2} .
\end{equation}
Second, ${n=k^{m_2(H)-o(1)}}$ and~$m_2(H)>1$ imply that there is~$\tau=\tau(H) >0$ such~that 
\begin{equation}\label{eq:aux2}
	\tfrac{k}{n} = k^{1-m_2(H)+o(1)} \ll k^{-\tau}/\log k .
\end{equation}
Third, using~${p = k^{-1+o(1)}}$ and strictly $2$-balancedness of~$H$ (implying that~${(e_J-1)/(v_J-2) < m_2(H)}$ for all~$J \subsetneq H$ with~${e_J \ge 2}$),
it follows that there is~$\gamma=\gamma(H) >0$ such~that 
\begin{equation}\label{eq:aux3}
	n^{v_J-2}p^{e_J-1} = \bigpar{n p^{(e_J-1)/(v_J-2)}}^{v_J-2} \gg k^{\gamma} \qquad \text{for all~$J \subsetneq H$ with~$e_J \ge 2$}.
\end{equation}
The below-claimed fourth estimate can be traced back to Erd{\H{o}}s and Tetali~\cite{erdos1990representations}; 
we include an elementary proof for self-containedness
(see~\cite[Section~2]{Warnke2017} for related estimates that also allow for overlapping edge-sets). 
\begin{claim}\label{cl:UT}
Let~$\cS$ be a collection of 
edge-subsets from~$E(K_n)$. 
Set~$\mu:= \sum_{\beta \in \cS}\Pr(\beta \subseteq E(G_{n,p}))$.
Define~$Z$~as~the largest number of disjoint edge-sets from~$\cS$ that are present in~$G_{n,p}$. 
Then~$\Pr(Z \geq x) \le (e\mu/x)^x$ for all~$x > \mu$.  
\end{claim}
\begin{proof}[Proof of \refCl{cl:UT}]
Set~$s := \ceil{x} \ge 1$. Exploiting edge-disjointness and~$s! \ge (s/e)^s$, it follows that 
\[
\Pr(Z \ge x) \le \sum_{\substack{\set{\beta_1, \ldots, \beta_s} \subseteq \cS:\\\text{all edge-disjoint}}}\underbrace{\Pr\bigpar{\beta_1 \cup \cdots \cup \beta_s \subseteq E(G_{n,p})}}_{= \prod_{1 \le i \le s}\Pr\xpar{\beta_i  \subseteq E(G_{n,p})}} \le \frac{1}{s!}\biggpar{\sum_{\beta \in \cS}\Pr\bigpar{\beta \subseteq E(G_{n,p})}}^s \le \bigpar{e \mu/s}^s, 
\]
which completes the proof by noting that the function~$s \mapsto (e \mu/s)^s$ is decreasing for positive~$s \ge \mu$. 
\end{proof}

We are now ready to bound the probability that~$|\cI_K|$ is large. 
Since~$H$ is strictly $2$-balanced, it contains no isolated vertices and thus is uniquely determined by its edge-set. 
This enables us to apply~\refCl{cl:UT} to~$|\cI_K|=Z$ (as~$\cI_K$ is a size-maximal collection of edge-disjoint $H$-copies from~$\cH^*_K$). 
Using estimate~\eqref{eq:aux1}, it is routine to see that, for small enough~${c \le c_0(C,\delta,H)}$, the associated parameter~$\mu$ from~\refCl{cl:UT} satisfies 
\begin{equation}\label{eq:IK:mu}
	\mu \le O\bigpar{k^2 n^{v_H-2} \cdot p^{e_H}} \le \tbinom{k}{2}p \cdot \Theta(n^{v_H-2} p^{e_H-1}) \le \tfrac{\delta}{2 e^2} \tbinom{k}{2}p .
\end{equation}
Noting~$\delta kp=\delta C \log k$ and $n \ll k^{e_H}$, now \refCl{cl:UT} (with~$Z=|\cI_K|$) implies that, for large enough~${C \ge C_0(\delta,H)}$, we have 
\begin{equation}\label{eq:IK}
	\Pr\Bigpar{|\cI_K| \ge \tfrac{\delta}{2}\tbinom{k}{2}p} 
	\le \biggpar{\frac{e\mu}{\tfrac{\delta}{2}\tbinom{k}{2}p}}^{\tfrac{\delta}{2}\tbinom{k}{2}p} 
	\le e^{-\tfrac{\delta}{2}\tbinom{k}{2}p} \ll k^{-e_H k} \ll n^{-k}. 
\end{equation}

Next, we similarly use \refCl{cl:UT} to bound the probability that~$|\cT_K|$ is large. 
For the associated parameter~$\mu$ we shall proceed similar to~\eqref{eq:IK:mu}: 
using estimates~\eqref{eq:aux1}--\eqref{eq:aux2}, for small enough~${c \le c_0(C,\delta,H)}$ we obtain 
\begin{equation}\label{eq:TK:mu}
	\mu \le  O\bigpar{k^3 n^{v_H-3} \cdot p^{e_H}} \le \tbinom{k}{2}p \cdot \tfrac{k}{n} \cdot \Theta(n^{v_H-2} p^{e_H-1}) \le k^{-\tau} \cdot \tfrac{\delta}{e} \tbinom{k}{2}p/\log k .
\end{equation}
With similar considerations as for~\eqref{eq:IK}, for large enough~${C \ge C_0(\tau,\delta,H)}$ \refCl{cl:UT} (with~$Z=|\cT_K|$) then yields
\begin{equation}\label{eq:TK}
	\Pr\Bigpar{|\cT_K|\ge \delta \tbinom{k}{2}p/\log k}\le k^{-\tau \delta\binom{k}{2}p/\log k} = e^{-\tau \delta \binom{k}{2}p}  \ll k^{-e_H k} \ll n^{-k}. 
\end{equation}

We shall analogously use \refCl{cl:UT} to bound the probability that~$|\cP_K|$ is large. 
For the associated parameter~$\mu$, 
the basic idea is to distinguish all possible subgraphs~$J \subsetneq H$ in which the relevant~${H_1,H_2 \in \cH_K^*}$ can intersect. 
Also taking into account the number of vertices which~$H_1$ and~$H_2$ have inside~$K$, i.e., ${\bigl|\bigpar{V(H_1) \cup V(H_2)} \cap K\bigr| \in \{3,4\}}$, 
by definition of~$\cP_K$ it now follows via estimates~\eqref{eq:aux1}--\eqref{eq:aux3} that
\begin{equation}\label{eq:PK:0}
	\begin{split}
		\mu & \le \sum_{J \subsetneq H: e_J \ge 1} O\Bigpar{k^3 n^{2(v_H-2)-(v_J-1)} \cdot p^{2 e_H-e_J} + k^4 n^{2(v_H-2)-v_J} \cdot p^{2 e_H-e_J}} \\
		& \le \tbinom{k}{2}p \cdot \Bigsqpar{\tfrac{k}{n} + \bigpar{\tfrac{k}{n}}^2} \cdot \sum_{J \subsetneq H: e_J \ge 1} \frac{\Theta\bigpar{(n^{v_H-2} p^{e_H-1})^2}}{n^{v_J-2}p^{e_J-1}} \le k^{-\tau} \cdot \tfrac{\delta}{e} \tbinom{k}{2}p/\log k .
	\end{split}
\end{equation}
(To clarify: in~\eqref{eq:PK:0} above we used that~\eqref{eq:aux3} implies~$n^{v_J-2}p^{e_J-1} \ge 1$ for all~$J \subsetneq H$ with~$e_J \ge 1$.)  
Similarly to inequalities~\eqref{eq:IK} and \eqref{eq:TK}, for large enough~${C \ge C_0(\tau,\delta,H)}$ now \refCl{cl:UT} (with~$Z=|\cP_K|$) 
yields 
\begin{equation}\label{eq:PK}
	\Pr\Bigpar{|\cP_K|\ge \delta \tbinom{k}{2}p/\log k} \le k^{-\tau\delta\binom{k}{2}p/\log k} = e^{-\tau \delta \binom{k}{2}p}  \ll k^{-e_H k} \ll n^{-k}. 
\end{equation}

Finally, combining~\eqref{eq:IK}, \eqref{eq:TK} and~\eqref{eq:PK} with Claim~\ref{claim:2}, 
a standard union bound argument gives
\begin{equation}\label{eq:Hremoval:0}
	\Pr\Bigpar{Y_K \ge \delta \tbinom{k}{2}p \cdot \bigpar{\tfrac{1}{2}+4e_H^2\Delta_H/\log k} \text{ for some $k$-vertex set~$K$}} \le \tbinom{n}{k} \cdot o(n^{-k}) = o(1).
\end{equation}
To complete the proof of the~\mbox{$Y_K$ bound}, it thus remains to show that, for small enough~${c \le c_0(C,H)}$, we~have 
\begin{equation}\label{eq:DeltaH}
	\Pr\Bigpar{\Delta_H \ge (\log k)/(8 e_H^2)}=o(1).
\end{equation} 
Using~\eqref{eq:aux1}, \eqref{eq:aux3} and~$n \ll k^{e_H}$, 
this upper tail estimate for~$\Delta_{H}=\max_f \Delta_{H,f}$ follows routinely from standard concentration inequalities such as~\cite[Theorem~32]{warnke2016missing}, 
but we include an elementary proof for self-containedness (based on ideas from~\cite{spohel2013general,Warnke2017}). 
Turning to the proof of~\eqref{eq:DeltaH}, 
let~$\Delta_{H,f,g}$ denote the number of $H$-copies in~$G_{n,p}$ that contain the edges~$\set{f,g}$, 
and define~$\Delta^{(2)}_{H}$~as the maximum of~$\Delta_{H,f,g}$ over all distinct~${f,g \in E(K_n)}$. 
We call an $r$-tuple~$(H_1,\dots, H_r)$ of~$H$-copies an~\emph{$(r,f,g)$-star} if  
each~$H_j$ contains the edges~$\set{f,g}$ and satisfies ${H_j \not\subseteq H_1 \cup \cdots \cup H_{j-1}}$. 
Define~$Z_{r,f,g}$ as the number of~$(r,f,g)$-stars~$(H_1,\dots, H_r)$ that are present in~$G_{n,p}$. 
Summing over all~$(r+1,f,g)$-stars~$(H_1,\dots, H_{r+1})$, 
by noting that the intersection of $H_{r+1}$ with ${F_r:=H_1 \cup \cdots \cup H_{r}}$ is isomorphic to some proper subgraph~${J \subsetneq H}$ containing at least~$e_J \ge 2$ edges, 
using estimates~\eqref{eq:aux1} and~\eqref{eq:aux3} it then is routine to see that, for~${1 \le r \le r_0:= 1+\ceil{(v_He_H + 4 e_H)/\gamma}}$, we~have 
\begin{equation*}
	\begin{split}
		\E Z_{r+1,f,g}  &= \sum_{(H_1, \ldots, H_{r+1})} p^{e_{H_1 \cup \cdots \cup H_{r+1}}} = \sum_{(H_1, \ldots, H_r)} p^{e_{F_r}} \sum_{H_{r+1}}p^{e_H-e_{H_{r+1} \cap F_r}} \\
		& \le \sum_{(H_1, \ldots, H_r)} p^{e_{F_r}} \cdot \sum_{J \subsetneq H: e_J \ge 2}O\Bigpar{(v_H r)^{v_J} n^{v_H-v_J} \cdot p^{e_H-e_J}} 
		\le \E Z_{r,f,g} \cdot k^{-\gamma} .
	\end{split}
\end{equation*}
Since trivially~$\E Z_{1,f,g} = O(n^{v_H})$, using~$n \ll k^{e_H}$ we infer~$\E Z_{r_0,f,g} \le k^{v_He_H-(r_0-1)\gamma} \le k^{-4 e_H} \ll n^{-4}$.  
Consider a maximal length~$(r,f,g)$-star~$(H_1,\dots, H_r)$ in~$G_{n,p}$, and 
note that in~$G_{n,p}$ any $H$-copy containing the edges~$\set{f,g}$ is completely contained in~$H_1 \cup \cdots \cup H_r$ (by length maximality), 
so that~$\Delta_{H,f,g} \le (e_H r)^{e_H}$ holds (using that~$H$ is uniquely determined by its edge-set). 
For~$D := (e_H r_0)^{e_H}$ it follows~that
\begin{equation}\label{eq:DeltaH2}
	\Pr\Bigpar{\Delta^{(2)}_H \ge D} \le \sum_{f \neq g}\Pr\bigpar{\Delta_{H,f,g} \ge D} \le \sum_{f \neq g} \Pr(Z_{r_0,f,g} \ge 1) \le \sum_{f \neq g} \E Z_{r_0,f,g} \le \tbinom{n}{2}^2 \cdot o(n^{-4}) = o(1). 
\end{equation} 
With an eye on~$\Delta_{H,f}$, let~$\cH_f$ denote the collection of all $H$-copies in~$K_n$ that contain the edge~$f$. 
We pick a subset~$\cI \subseteq \cH_f$ of $H$-copies in~$G_{n,p}$ that is size-maximal subject to the restriction that all $H$-copies are edge-disjoint after removing the common edge~$f$. 
For any~$H' \in \cH_f$, note that in~$G_{n,p}$ there are a total of 
at most~$e_H \Delta^{(2)}_H$ copies of~$H$ that share~$f$ and at least one additional edge with~$H'$. 
Hence~$\Delta_{H,f} \ge (\log k)/(8 e_H^2)$ and~$\Delta^{(2)}_H \le D$ imply~$|\cI| \ge \ceil{(\log k)/A}=:z$ for~$A := 8 e_H^3 D$ (by maximality of~$\cI$).  
As the union of all $H$-copies in~$\cI$ contains exactly $1+(e_H-1)|\cI|$ edges, 
using~$\tbinom{m}{z} \le (em/z)^z$ and~$|\cH_f| =O(n^{v_H-2})$ it follows~that 
\begin{equation}\label{eq:DeltaHf}
	\Pr\Bigpar{\Delta_{H,f} \ge (\log k)/(8 e_H^2) \text{ \ and \ } \Delta^{(2)}_H \le D} \le \binom{|\cH_f|}{z} \cdot p^{1+(e_H-1)z} \le \biggpar{\frac{O(n^{v_H-2}p^{e_H-1})}{z}}^{z} .
\end{equation}
Using estimate~\eqref{eq:aux1}, for small enough~${c \le c_0(A,C,H)}$ the right-hand side of~\eqref{eq:DeltaHf} is at most~$(\log k)^{-(\log k)/A} \ll k^{-2 e_H}$.  
Recalling~$n \ll k^{e_H}$, by taking a union bound over all edges~$f \in E(K_n)$ it then follows that 
\begin{equation}\label{eq:DeltaH:UB}
	\Pr\Bigpar{\Delta_H \ge (\log k)/(8 e_H^2)  \text{ \ and \ } \Delta^{(2)}_H \le D} 
	\le \tbinom{n}{2} \cdot o(k^{-2 e_H}) = o(1) , 
\end{equation} 
which together with~\eqref{eq:DeltaH2} completes the proof of estimate~\eqref{eq:DeltaH} and thus the~\mbox{$Y_K$ bound} of \refT{thm:Hremoval}. 
\end{proof}
%
The above proof of~\eqref{eq:DeltaH} can easily be sharpened 
to ${\Pr\bigpar{\Delta_H \ge B(\log k)/\log \log k}=o(1)}$ for a suitable constant ${B=B(H)>0}$, see~\eqref{eq:DeltaHf}--\eqref{eq:DeltaH:UB}. 
Together with the proof of~\eqref{eq:Hremoval:0} and~${|\cI_K| \le Y_K}$, 
this implies that whp ${Y_K = |\cI_K|+o\bigpar{\delta\tbinom{k}{2}p}}$ for all $k$-vertex~subsets~$K$, 
which suggests that~$Y_K$ is well-approximated by~$|\cI_K|$.

To complete the proof of \refT{thm:Hremoval}, it remains to give the routine Chernoff bound based proof of the \mbox{$X_K$ bound} of \refT{thm:Hremoval}, 
where~$X_K=|E(G_{n,p}[K])|$ denotes the number of edges of~$G_{n,p}$ inside~$K$. 
\begin{proof}[Proof of the~\mbox{$X_K$ bound} of \refT{thm:Hremoval}]
Noting~$X_K \sim\mathrm{Bin}\bigpar{\tbinom{k}{2},p}$ as well as~$\delta^2kp = \delta^2C\log k$ and $n \ll k^{e_H}$, 
by standard Chernoff bounds (such as~\cite[Theorem~2.1]{JLR}) it follows, for large enough~${C \ge C_0(\delta,H)}$, that
\begin{equation*}
	\Pr\bigpar{X_K \le (1-\delta) \tbinom{k}{2}p} \;\le\; \indic{\delta \in (0,1]}\exp\Bigpar{-\delta^2\tbinom{k}{2}p/2} \ll k^{-e_H k} \ll n^{-k}. 
\end{equation*}
Taking a union bound over all $k$-vertex subsets~$K$ completes the proof of the~\mbox{$X_K$ bound} of \refT{thm:Hremoval}. 
\end{proof}

\section{Extensions}\label{sec:extensions}
In applications of the alteration method outlined in \refS{sec:intro:tool}, it often is beneficial to keep track of further properties of the resulting $H$-free $n$-vertex graph~$G \subseteq G_{n,p}$, 
including vertex-degrees and the number of edges (see, e.g., \cite[Section~3]{erdos1988}, \cite[Section~2]{bohman2010coloring}, and~\cite[Section~5.1]{KLST2018}). 
Using the arguments and intermediate results from \refS{sec:Alteration}, oftentimes it is routine to show that~$G$ resembles a random graph~$G_{n,p}$ in many ways. 
For example, with standard results for~$G_{n,p}$ in mind, the following simple lemma intuitively implies that whp  
the resulting~$G$ is approximately~$np$ regular, with about~$\tbinom{n}{2}p$ edges 
(note that~$np \gg 1$ when~$m_2(H) > 1$). 
\begin{lemma}\label{lem:degree}
Let~$H$ be a strictly~$2$-balanced graph with~$m_2(H)>1$.   
Define~$Y$ as the number of~$H$-copies in~$G_{n,p}$, 
and define~$Y_v$ as the number of~$H$-copies in~$G_{n,p}$ that contain the vertex~$v$. 
For any fixed~$\delta > 0$, the following holds for all sufficiently large~${C \ge C_0=C_0(\delta,H)}$ and sufficiently small~${0 < c \le c_0=c_0(C,\delta,H)}$. 
Setting~$n$ and~$p$ as in~\refT{thm:Hremoval}, whp~$G_{n,p}$ satisfies~$Y_v \le \delta np$ for all vertices~$v$, and~$Y \le \delta \tbinom{n}{2}p$. 
\end{lemma}
\begin{proof}
Since~$m_2(H)>1$ implies~$v_H \ge 3$, noting~$Y = \sum_{v \in [n]} Y_v/v_H$ it suffices to prove the claimed bounds on the~$Y_v$. 
Fix a vertex~$v$.
Similar to estimate~\eqref{eq:IK:mu}, using~\eqref{eq:aux1} it is standard to see that the expected number of $H$-copies containing~$v$ 
is at most $\mu \le O(n^{v_H-1}p^{e_H}) \le \tfrac{\delta}{e^2}np$ for small enough~${c \le c_0(C,\delta,H)}$.  
Furthermore, if~$\Delta_H \le (\log k)/(8 e_H^2)$ holds (see~\eqref{eq:DeltaH} in \refS{sec:Alteration}), 
then any $H$-copy edge-intersects a total of at most~${e_H \cdot \Delta_H < \log k}$ many $H$-copies, say. 
Applying the upper tail inequality~\cite[Theorem~15]{GW2017} instead of~\refCl{cl:UT}, 
using~$\delta np=\delta cC k^{m_2(H)-1-o(1)} \gg (\log k)^2$ it then is, similar to~\eqref{eq:IK} and~\eqref{eq:DeltaH:UB}, routine to see~that
\begin{equation*}
	\Pr\Bigpar{Y_v \ge \delta np  \text{ \ and \ } \Delta_H \le (\log k)/(8 e_H^2)} 
	\le \biggpar{\frac{e\mu}{\delta np}}^{\delta np/\log k} 
	\le e^{-\delta np/\log k} \ll n^{-1}. 
\end{equation*}
Taking a union bound over all vertices~$v$ now completes the proof together with estimate~\eqref{eq:DeltaH}. 
\end{proof}

It is straightforward, and useful for many applications (see, e.g.,~\cite{krivelevich1997approximate,AF2008,bohman2010coloring}), 
to extend the alteration method to $r$-uniform hypergraphs, where every edge contains $r \ge 2$ vertices. 
Indeed, to forbid a given $r$-uniform hypergraph~$H$, similarly to the graph case~($r=2$) discussed in \refS{sec:intro:tool}, 
here the idea is to delete edges from a \emph{random $r$-uniform hypergraph~$G^{(r)}_{n,p}$} 
(where each of the $\tbinom{n}{r}$ possible edges appears independently with probability~$p$) 
to construct an $n$-vertex $r$-uniform hypergraph~${G \subseteq G^{(r)}_{n,p}}$ that is $H$-free. 
Defining 
\[
m_r(H):=\max_{F\subseteq H}\Bigpar{\indic{v_F \ge r+1}\tfrac{e_F-1}{v_F-r} + \indic{v_F=r,e_F=1}\tfrac{1}{r}} ,
\]
we say that~$H$ is \emph{strictly~$r$-balanced} if $m_r(H) > m_r(F)$ for all~$F \subsetneq H$. 
Noting~$G_{n,p}=G^{(2)}_{n,p}$, now the proof of \refT{thm:Hremoval}
routinely carries over with only obvious notational changes (where~$X_K=|E(G^{(r)}_{n,p}[K])|$ denotes the number of edges of~$G^{(r)}_{n,p}$ inside~$K$, 
and~$Y_K$ denotes the number of edges in~$E(G^{(r)}_{n,p}[K])$ that are in some~$H$-copy of~$G_{n,p}$), 
yielding the following hypergraphs extension of our main alteration tool \refT{thm:Hremoval}.  
%
\begin{theorem}\label{thm:Hremoval:HG}
Given~$r \ge 2$, let~$H$ be a strictly~$r$-balanced $r$-uniform hypergraph. 
Then, for any fixed ${\delta >0}$, the following holds for all sufficiently large~${C \ge C_0=C_0(\delta,H,r)}$ and sufficiently small~${0 < c \le c_0=c_0(C,\delta,H,r)}$. 
The random $r$-uniform hypergraph~$G^{(r)}_{n,p}$ with ${n := \floor{c (k/\log k)^{m_r(H)}}}$ vertices and edge-probability ${p := C(\log k)/k}$ 
whp 
satisfies~${Y_K \le \delta \tbinom{k}{r}p}$ and~${X_K \ge (1-\delta) \tbinom{k}{r}p}$ for all~$k$-vertex subsets~$K$ of~$G^{(r)}_{n,p}$. 
\end{theorem}

Finally, numerous applications~\cite{krivelevich1997approximate,krivelevich1997minimal,AF2008,bohman2010coloring} of the alteration method require forbidding a finite collection of hypergraphs~$\cH={\{H_1, \ldots, H_s\}}$. 
The crux is that the bounds on~$Y_K$ and~$X_K$ from \refT{thm:Hremoval:HG} trivially remain valid for~${n \le \floor{c (k^{r-1}/\log k)^{m_r(H)}}}$.  
So, applying this result to all forbidden~${H_i \in \cH}$ simultaneously, 
one can easily obtain a variant of \refT{thm:Hremoval:HG} where~$Y_K$ denotes the number of edges in~$G^{(r)}_{n,p}[K]$ 
that are in at least one~$H_i$-copy of~$G_{n,p}^{(r)}$ for some~$H_i \in \cH$; 
we leave the routine details to the interested~reader. 
%
%

\bigskip{\noindent\bf Acknowledgements.} 
We would like to thank Jacob Fox for helpful clarifications regarding~\cite{conlon2018online}. 
We also thank the referees for useful comments concerning the presentation.

\small
\bibliographystyle{plain}

\vspace{-0.85em}

\normalsize

\begin{appendix}
	
\section{Appendix: Lower bound on the upper tail of~$|\cH_K|$}\label{apx:UT}
Given a fixed graph~$H$ with~$v_H \ge 3$, let us consider a binomial random graph~$G_{n,p}$ with edge-probability ${p = \Theta\bigpar{(\log k)/k}}$ as~$k \to \infty$. 
Fix a $k$-vertex subset~$K$ of~$G_{n,p}$ (which tacitly requires~$k \le n$), 
and let~$\cH_K$ denote the collection of all $H$-copies that have at least one edge inside~$K$. 
Given~$\delta>0$, we fix~$v_H$ disjoint vertex subsets of~$K$, each of size~${t:=\bigceil{\bigpar{\delta \tbinom{k}{2}p}^{1/v_H}}}$. 
Then~$G_{n,p}$ contains with probability~$p^{\tbinom{v_H}{2}t^2}$ a complete $v_H$-partite subgraph on these $v_H$~sets, 
which enforces~${|\cH_K| \ge t^{v_H} \ge \delta \tbinom{k}{2}p}$. 
It readily follows that 
\[
\Pr\Bigpar{|\cH_K| \ge \delta \tbinom{k}{2}p} \;\ge\; p^{\tbinom{v_H}{2}t^2} \;\ge\; e^{-o(k)} ,
\] 
as claimed in~\refS{sec:intro:tool} (since~${t^2 \cdot \log(1/p) \le k^{2/v_H + o(1)} \cdot O(\log k) = o(k)}$ as~$k \to \infty$). 
	
\end{appendix}
	
\end{document}